\documentclass[12pt,a4paper,twoside]{article}
\usepackage{a4}
\usepackage{amsmath,amsfonts,amssymb,amsthm,amscd}
\usepackage{epsfig}

\usepackage{hyperref}

\newtheorem{theorem}{Theorem}                            
\newtheorem{proposition}[theorem]{Proposition}            
\newtheorem{lemma}[theorem]{Lemma}

\begin{document}

\date{}
\title{Ramsey-type constructions for arrangements of segments}

\author{Jan Kyn\v cl\thanks{%
The author was partially supported by the Phenomena in High Dimensions project, in the framework of the European Community's ``Structuring the European Research Area'' program.}\\
}% end of thanks and author
     
\maketitle

\vskip -2cm

\begin{center}
{\small
Department of Applied Mathematics and\\
Institute for Theoretical Computer Science (ITI)\footnote{ITI is supported by project 1M0545 of the Ministry of  Education of the Czech Republic.}\\
Charles University, Faculty of Mathematics and Physics\\
Malostransk\'e n\'am.~25, 118~ 00 Prague, Czech Republic\\
\tt{kyncl@kam.mff.cuni.cz}
}
\end{center}

\bigskip

\begin{abstract}
Improving a result of K\'arolyi, Pach and T\'oth, we construct an arrangement of $n$ segments in the plane with at most $n^{\log{8} / \log{169}}$ pairwise crossing or pairwise disjoint segments. We use the recursive method based on flattenable arrangements which was established by Larman, Matou\v{s}ek, Pach and T\"or{\H{o}}csik. We also show that not every arrangement can be flattened, by constructing an intersection graph of segments which cannot be realized by an arrangement of segments crossing a common line. Moreover, we also construct an intersection graph of segments crossing a common line which cannot be realized by a flattenable arrangement. 
\end{abstract}

%=========================1. INTRO============================

\section{Introduction}
An {\em arrangement of segments\/} is a finite set of compact straight-line segments in the plane in general position (i.e., no three endpoints are collinear). We study the following Ramsey-type problem~\cite{LMPT94_a_ramsey}: what is the largest number $r(k)$ such that there exists an arrangement of $r(k)$ segments with at most $k$ pairwise crossing and at most $k$ pairwise disjoint segments? 

Larman et al.~\cite{LMPT94_a_ramsey} proved that
$k^5 \ge r(k) \ge k^{\log{5}/\log{2}} > k^{2.3219}$. The upper bound has remained unchanged since then. K\'{a}rolyi et al.~\cite{KPT97_ramsey} improved the lower bound to $r(k) \ge k^{\log{27}/\log{4}} > k^{2.3774}$.

We improve the construction for the lower bound even further and prove the following theorem.

\begin{theorem}\label{veta_hlavni}
For infinitely many positive integers $k$ there exists an arrangement of $k^{\log{169}/\log{8}} > k^{2.4669}$ segments with at most $k$ pairwise crossing and at most $k$ pairwise disjoint segments.
\end{theorem} 
%inverze 0.405357232

Similar questions were studied by Fox, Pach and Cs. T\'oth~\cite{FPT_intersection} for string graphs, a class of graphs generalizing intersection graphs of segments. They proved, as a consequence of a stronger result, that for each positive integer $k$ there is a constant $c(k)>0$ such that in any system of $n$ curves in the plane where every two curves intersect in at most $k$ points, there is a subset of $n^{c(k)}$ curves that are pairwise disjoint or pairwise crossing. 

\section{Proof of Theorem 1} 

Both previous constructions for the lower bound~\cite{KPT97_ramsey, LMPT94_a_ramsey} use the same approach. The starting configuration is an arrangement $M_0$ of $n_0$ segments with at most $k_0$ pairwise crossing or pairwise disjoint segments. In the $i$-th step, an arrangement $M_{i}$ of $n_0^{i+1}$ segments is constructed from the arrangement $M_{i-1}$ by replacing each of its segments by a flattened copy (a precise definition will follow) of $M_{0}$, which acts as a ``thick segment''. Then two segments from different copies of $M_{0}$ cross if and only if the two corresponding segments in $M_{i-1}$ cross. Our new arrangement $M_{i}$ has then at most $k_0^{i+1}$ pairwise crossing or pairwise disjoint segments. This gives a lower bound $r(k) \ge k^{\log{n_0}/\log{k_0}}$ for infinitely many values of $k$.

We improve the construction by making a better starting arrangement. 
Unlike the previous constructions, our basic pieces will be arrangements with different maximal numbers of pairwise crossing and pairwise disjoint segments. By putting them together, we obtain our starting arrangement $M_0$.

Let ${\rm Cay}({\mathbb Z}_{13};1,5)$ denote the Cayley graph of the cyclic group ${\mathbb Z}_{13}$ corresponding to the generators $1$ and $5$. That is, $V({\rm Cay}({\mathbb Z}_{13};1,5))=\{1,2, \dots, 13\}$ and $E({\rm Cay}({\mathbb Z}_{13};1,5))=\{\{i,j\};1\le i < j \le 13, (j-i) \in \{1,5,8,12\}\}$. See Figure~\ref{obr_2_1_cayley}.

\begin{figure}
\begin{center}
\epsfbox{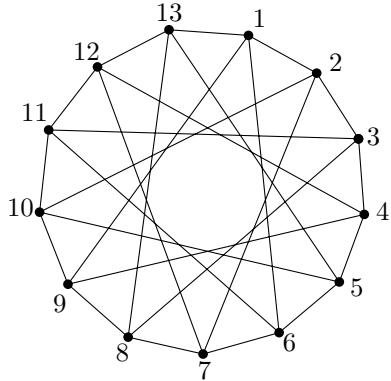}
\end{center}
\caption{A Cayley graph ${\rm Cay}({\mathbb Z}_{13};1,5)$.}
\label{obr_2_1_cayley}
\end{figure}

\begin{lemma}
The graph ${\rm Cay}({\mathbb Z}_{13};1,5)$ contains no clique of size $3$ and no independent set of size $5$.
\end{lemma}

\begin{proof} 
Suppose that $a<b<c$ are three vertices of ${\rm Cay}({\mathbb Z}_{13};1,5)$ inducing a clique. Then the numbers $k=c-a, l=c-b$ and $m=b-a$ belong to the set $\{1,5,8,12\}$, but this set contains no triple $k,l,m$ satisfying the equation $k=l+m$; a contradiction.

Now suppose that $A=\{a<b<c<d<e\}$ is an independent set of ${\rm Cay}({\mathbb Z}_{13};1,5)$. By the pigeon-hole principle, $A$ contains two vertices with difference $2$ (modulo $13$). Thus we can without loss of generality assume that $a=1$ and $b=3$. It follows that $\{c,d,e\} \subseteq \{5,7,10,12\}$. But $A$ cannot contain both $5$ and $10$, neither both $7$ and $12$. Hence $|A \cap \{5,7,10,12\}|\le 2$; a contradiction.
\end{proof}

A {\em $(k,l)$-arrangement\/} is an arrangement of segments with at most $k$ pairwise crossing and at most $l$ pairwise disjoint segments.

An {\em intersection graph} $G(M)$ of an arrangement $M$ is a graph whose vertices are the segments of $M$ and two vertices are joined by an edge if and only if the corresponding segments intersect. 

An arrangement $M$ of segments is {\em flattenable\/} if for every $\varepsilon > 0$ there is an arrangement $M_{\varepsilon}$ with $G(M_{\varepsilon})=G(M)$ and two discs $D_1, D_2$ of radius $\varepsilon$ whose centers are at unit distance, such that each segment from $M_{\varepsilon}$ has one endpoint in $D_1$ and the second endpoint in $D_2$. A {\em flattened copy\/} of $M$ is the arrangement $M_{\varepsilon}$ with sufficiently small ${\varepsilon}$.

The key result is the following lemma.

\begin{lemma}
\begin{enumerate}
\item There exists a flattenable $(2,4)$-arrangement of $13$ segments.
\item There exists a flattenable $(4,2)$-arrangement of $13$ segments.
\end{enumerate}
\end{lemma}

Note that $13$ is the largest possible number of segments for these two types of arrangements since every graph with more than $13$ vertices contains either a clique of size $5$ or an independent set of size $3$~\cite{R94_small}.

Both previous constructions~\cite{KPT97_ramsey, LMPT94_a_ramsey} used {\em convex\/} starting arrangement, i.e., an arrangement of segments with endpoints in convex position. Convex arrangements are flattenable by a relatively simple argument~\cite{KPT97_ramsey}. However, Kostochka~\cite{K88_upper} proved that any convex $(k,k)$-arrangement has at most $(1+o(1))\cdot k^2\log k$ segments. He also gave a construction of a convex $(k,k)$-arrangement with $\Omega(k^2\log k)$ segments (see also~\cite{C08_circle}). \v{C}ern\'y~\cite{C08_circle} investigated convex $(k,l)$-arrangements for small values of $k$. He showed, in particular, that any convex $(2,4)$-arrangement has at most $12$ segments, and that any convex $(4,2)$-arrangement has at most $11$ segments.

Our starting arrangements thus cannot be convex. Hence their flattening will require a special approach.

\begin{proof}
For each sufficiently small $\varepsilon >0$, we construct an arrangement $M_a(\varepsilon)$ with intersection graph ${\rm Cay}({\mathbb Z}_{13};1,5)$ and an arrangement $M_b(\varepsilon)$ whose intersection graph is the complement of ${\rm Cay}({\mathbb Z}_{13};1,5)$. See Figure~\ref{obr_1_2442} for an illustration.

In Tables~\ref{tab_M_a} and~\ref{tab_M_b}, we provide precise coordinates of the endpoints of all the $13$ segments, as functions of $\varepsilon$. To achieve general position of the segments, which is required by our definition, we can slightly perturb the endpoints while preserving the intersection graph of the arrangement.

\begin{table}
\begin{center}
\begin{tabular}{l|l|l|l|l}  
     &     left $x$ & left $y$ & right $x$ & right $y$ \\
\hline
$1$  &  $-\varepsilon$     &   $0$        &    $1-2\varepsilon$    &  $2\varepsilon^2+2\varepsilon^6$ \\
$2$  &  $\varepsilon^2$    &   $\varepsilon-\varepsilon^3$    &    $1-\varepsilon^2$   &    $\varepsilon^3$ \\
$3$  &   $0$     &   $\varepsilon^4+\varepsilon^6$  &    $1$       &  $\varepsilon^3+3\varepsilon^4$ \\
$4$  &   $0$     &   $\varepsilon^4-\varepsilon^6$  &    $1-2\varepsilon$    &  $2\varepsilon^2-\varepsilon^6$ \\
$5$  & $-\varepsilon+\varepsilon^2$  &   $0$        &    $1-2\varepsilon^2$  &  $2\varepsilon^3-2\varepsilon^4$ \\
$6$  &  $-\varepsilon$     &   $2\varepsilon^6$     &    $1-\varepsilon$     &   $2\varepsilon^6$ \\
$7$  &   $0$     &   $\varepsilon^6$      &    $1$       &  $\varepsilon^3+2\varepsilon^4$ \\
$8$  &   $0$     &   $\varepsilon$        &    $1+\varepsilon^3$   &     $0$ \\
$9$  &   $0$     &   $\varepsilon$        &    $1-2\varepsilon^2$  &  $2\varepsilon^3-\varepsilon^4$ \\
$10$ & $-\varepsilon^2+3\varepsilon^3$&   $3\varepsilon^6$     &    $1-2\varepsilon$    &  $2\varepsilon^2+\varepsilon^6$ \\
$11$ & $-\varepsilon^2$    &   $\varepsilon^6$      &    $1-2\varepsilon^2$  &  $2\varepsilon^3-3\varepsilon^4$ \\
$12$ &   $0$     &   $\varepsilon^4$      &    $1$       &     $0$ \\
$13$ &  $-\varepsilon$     &   $0$        &    $1+\varepsilon$     &     $0$ \\
\end{tabular}
\end{center}
\caption{Arrangement $M_a(\varepsilon)$.}
\label{tab_M_a}
\end{table}

\begin{table}
\begin{center}
\begin{tabular}{l|l|l|l|l}  
     &     left $x$ & left $y$ & right $x$ & right $y$ \\
\hline
$1$   & $\varepsilon$ &         $\varepsilon^2-\varepsilon^3+\varepsilon^4-2\varepsilon^5$  &   $1+\varepsilon^2$ &        $-\varepsilon^4+\varepsilon^6$ \\
$2$   & $0$ &                   $\varepsilon^2+3\varepsilon^5$  &             $1-\varepsilon^3$ &         $\varepsilon^7$ \\
$3$   & $0$ &                   $\varepsilon^2+4\varepsilon^5$  &             $1+\varepsilon$ &          $-\varepsilon^3$ \\
$4$   & $0$ &                   $2\varepsilon^3$  &                 $1+3\varepsilon^4$ &       $-\varepsilon^8$ \\
$5$   & $\varepsilon-\varepsilon^2+\varepsilon^3$ &           $\varepsilon^2-\varepsilon^3+\varepsilon^4-\varepsilon^8$ &      $1+\varepsilon$ &    $-\varepsilon^4$ \\
$6$   & $0$ &                   $\varepsilon^2+\varepsilon^5$ &              $1+\varepsilon$ &          $-\varepsilon^3$ \\
$7$   & $0$ &                   $\varepsilon^2+5\varepsilon^5$ &             $1+3\varepsilon^4$ &       $-3\varepsilon^7$ \\
$8$   & $\varepsilon-\varepsilon^2+\varepsilon^3+\varepsilon^4+2\varepsilon^5$ &  $\varepsilon^2-\varepsilon^3+\varepsilon^4+\varepsilon^5+\varepsilon^6$ &  $1+\varepsilon-\varepsilon^4$ &      $-\varepsilon^3$ \\
$9$   & $0$ &                   $\varepsilon^2$ &        $1+\varepsilon$ &          $-\varepsilon^4$ \\
$10$  & $0$ &                   $0$ &                    $1+5\varepsilon^3$ &        $0$ \\
$11$  & $0$ &                   $\varepsilon^2+2\varepsilon^5$ &             $1+3\varepsilon^4-2\varepsilon^5$ &   $\varepsilon^8$ \\
$12$  & $\varepsilon-\varepsilon^3$ &               $\varepsilon^3-\varepsilon^4$ &              $1+\varepsilon$ &          $-\varepsilon^4$ \\
$13$  & $0$ &                   $0$ &               $1$ &             $\varepsilon$ \\
\end{tabular}
\end{center}
\caption{Arrangement $M_b(\varepsilon)$.}
\label{tab_M_b}
\end{table}

\begin{figure}
\begin{center}
\epsfbox{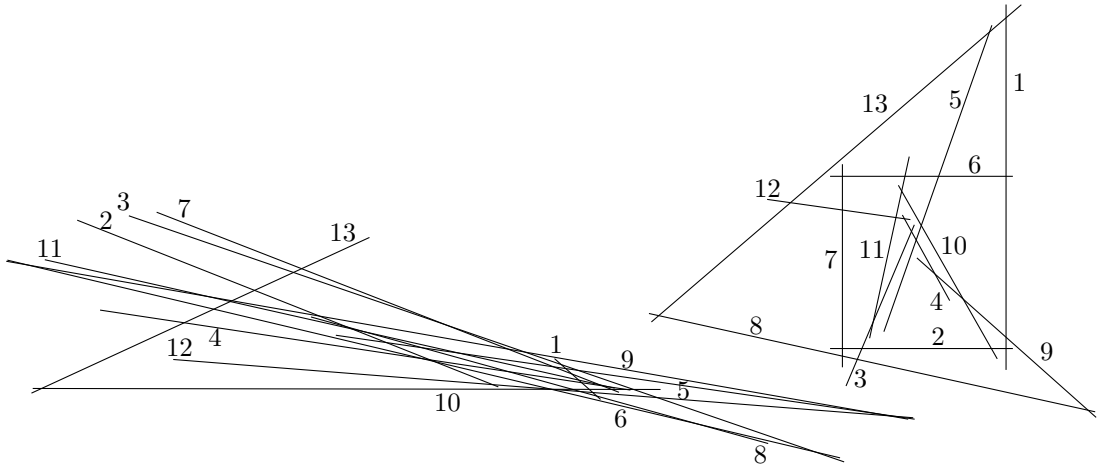}
\end{center}
\caption{A partially flattened $(4,2)$-arrangement of 13 segments (left) and a $(2,4)$-arrangement of 13 segments (right).}
\label{obr_1_2442}
\end{figure}

Since the coordinates of all the left endpoints converge to $(0,0)$ and the coordinates of all the right endpoints converge to $(1,0)$, it remains to verify that for sufficiently small $\varepsilon > 0$, each of these two described arrangements has the desired intersection graph. This is a straightforward calculation, which can be done by the following simple algorithm.

We use the fact that the functions describing the coordinates are polynomials in $\varepsilon$. For $i \in {1, 2, \dots, 13}$, let $s_i$ be the $i$-th segment of the arrangement and let $l_x(i), l_y(i), r_x(i), r_y(i)$ be the polynomials representing the coordinates of the left and the right endpoint of $s_i$. For each pair $i<j$, we need to determine whether $s_i$ and $s_j$ cross if $\varepsilon$ is small enough.

Let $s$ be a segment with endpoints $(l_x,l_y)$ and $(r_x,r_y)$ and let $s'$ be a segment with endpoints $(l'_x,l'_y)$ and $(r'_x,r'_y)$. Let $p$ be the line containing $s$, and let $p'$ be the line containing $s'$. The segments $s$ and $s'$ intersect if and only if $s' \cap p \neq \emptyset$ and $s \cap p' \neq \emptyset$. We have $p=\{(x,y); ax+by+c=0 \}$, where $a=r_y-l_y$, $b=r_x-l_x$ and $c= r_xl_y - l_xr_y$. Thus, $s' \cap p \neq \emptyset$ if and only if $(al'_x+bl'_y+c)(ar'_x+br'_y+c)\le 0$. The relation $s \cap p' \neq \emptyset$ can be expressed similarly. 

The algorithm now follows. For each $i$, compute the polynomials $a_i=r_y(i)-l_y(i)$, $b_i=r_x(i)-l_x(i)$ and $c_i= r_x(i)l_y(i) - l_x(i)r_y(i)$. Then for each pair $i\neq j$, compute the polynomial $d_{i,j}=(a_il_x(j)+b_il_y(j)+c_i)(a_ir_x(j)+b_ir_y(j)+c_i)$. Now $s_i$ and $s_j$ intersect if and only if each $d_{i,j}$ and $d_{j,i}$ is nonpositive in some positive neighborhood of $0$. That is, the polynomial is either zero or the coefficient by the non-zero term of the smallest order is negative.

A program verifying both constructions can be downloaded from the following webpage: \href{http://kam.mff.cuni.cz/~kyncl/programs/segments}{http://kam.mff.cuni.cz/\~{}kyncl/programs/segments}. 
\end{proof}

Now we are ready to finish the proof of Theorem~\ref{veta_hlavni}.
Take a sufficiently flattened arrangement $M_a(\varepsilon)$ and replace each of its segments by a copy of a sufficiently flattened arrangement $M_b(\delta)$. In this way we obtain our starting flattenable $(8,8)$-arrangement $M_0$ of $169$ segments. Then we proceed by the method described at the beginning of this section.

\section{Non-flattenable arrangements}

Since the flattenable arrangements are the main tool in the construction in the previous section, it is natural to ask whether every arrangement of segments can be flattened. A necessary condition for an arrangement to be flattenable is the existence of a line crossing all the segments, in a sufficiently flattened realization. We show the following.

\begin{theorem}\label{veta_ne_primkou}
There exists an intersection graph of segments which cannot be realized by an arrangement of segments crossing a common line. 
\end{theorem}

\begin{theorem}\label{veta_crossable_nonflat}
There exists an arrangement of segments crossing a common line which is not flattenable.
\end{theorem}

\subsection{Proof of Theorem~\ref{veta_ne_primkou}}

Let $G$ be an intersection graph of the arrangement in Figure~\ref{obr_3_noncross}. The arrangement consists of $7$ {\em horizontal\/} and $7$ {\em vertical\/} segments forming a grid, the $56$ {\em frame\/} segments forming a cycle, $28$ {\em joining\/} segments connecting a grid segment with a segment of the frame (each grid segment is joined to the frame by two joining segments and every other segment from the frame is used), and finally $8$ {\em short\/} segments, each crossing one vertical and one horizontal segment from the grid. 

\begin{figure}
\begin{center}
\epsfbox{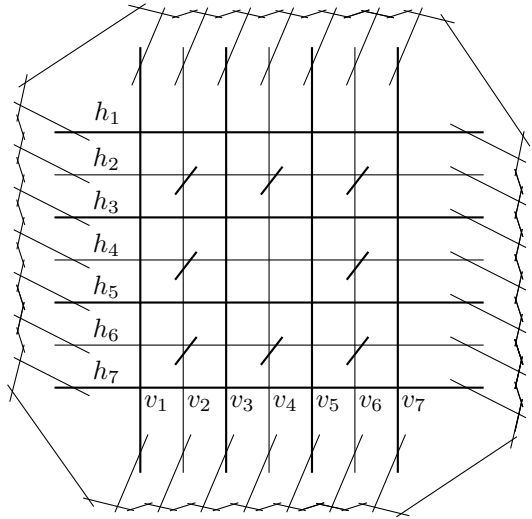}
\end{center}
\caption{A construction for Theorem~\ref{veta_ne_primkou}.}
\label{obr_3_noncross}
\end{figure}

We prove Theorem~\ref{veta_ne_primkou} in a slightly stronger form.

An {\em arrangement of pseudosegments\/} is a set of simple curves in the plane such that every two of the curves have at most one common point and any such point is a proper crossing. If $M$ is an arrangement of pseudosegments, then each curve from $M$, and also any curve $c$ such that $M \cup \{c\}$ is an arrangement of pseudosegments, is called a {\em pseudosegment}. 

\begin{proposition}
For any arrangement $M$ of pseudosegments whose intersection graph is $G$, no pseudosegment can cross all the curves from $M$.
\end{proposition}

\begin{proof}
Let $M$ be an arrangement of pseudosegments whose intersection graph is $G$. We use the terms {\em frame/grid/horizontal/vertical/joining/short\/} pseudosegment in a similar meaning as above. The union $\Gamma$ of the frame pseudosegments contains a unique closed curve $\gamma$. Each frame pseudosegment intersects $\gamma$ in a connected arc and the cyclic order of these arcs along $\gamma$ is uniquely determined (up to inversion). Both $\Gamma$ and $\gamma$ cut the plane into two connected regions. Since the subgraph of $G$ induced by the grid and short vertices is connected and separated from the frame cycle by the joining vertices, the union of the grid and short pseudosegments is connected and disjoint from $\Gamma$. Thus we can without loss of generality assume that all the grid and short pseudosegments lie in the region $\Omega$ bounded by $\Gamma$.

The order of the intersections of the joining pseudosegments with $\Gamma$ along the boundary of $\Omega$ is uniquely determined. Each grid pseudosegment together with its two joining pseudosegments divides $\Omega$ into two connected components. All the $7$ vertical pseudosegments with their joining pseudosegments divide $\Omega$ into $8$ connected components and the ``horizontal'' order of the vertical pseudosegments is uniquely determined. Each horizontal pseudosegment has to start in the leftmost region and end in the rightmost region and is forced to cross the vertical pseudosegments in the same order and orientation. Similarly each vertical pseudosegments has to cross all the horizontal pseudosegments in the same order and orientation. It follows that the grid pseudosegments form a ``pseudogrid'' homeomorphic to the grid in Figure~\ref{obr_3_noncross}. Therefore, we can further assume that the grid pseudosegments are straight-line segments forming a regular square grid.

Label the vertical and the horizontal segments of the grid consecutively by $v_1, v_2, \dots v_7$ and $h_1, h_2, \dots h_7$. The odd-numbered segments form a {\em coarse\/} grid of $3 \times 3$ {\em big\/} squares. Each of the eight short pseudosegments is contained in one of the big squares, since for each orthogonal pair $v_{2i}, h_{2j}$ of the even-labeled grid segments, the big square determined by the segments $v_{2i-1}, v_{2i+1}, h_{2j-1}$ and $h_{2j+1}$ is the only face in the arrangement $M \setminus (\{ v_{2i}, h_{2j} \} \cup M_s)$ intersected by both $v_{2i}$ and $h_{2j}$ (here $M_s$ denotes the set of short pseudosegments in $M$).
Therefore, each pseudosegment that crosses all pseudosegments in $M$ must intersect at least $8$ big squares in the coarse grid. We show that such pseudosegment does not exist.

Let $p$ be a pseudosegment. Suppose that $p$ has both its endpoints outside the grid. Then $p$ can enter and leave the grid at most twice, since in each traversal of the grid $p$ crosses two of the four boundary segments $v_1,v_7,h_1,h_7$. If $p$ intersects $k$ big squares in a traversal, it has to cross at least $k+1$ segments of the coarse grid (including two of the boundary segments). It follows that $p$ can intersect at most $5$ big squares in one traversal, and at most $6$ big squares in two traversals. 

Now suppose that $p$ starts outside and ends inside the grid. Suppose further that $p$ intersects $k$ big squares during the first traversal and then $l$ big squares after entering the grid for the second time. Then $p$ has to cross at least $k+1+l$ coarse grid segments. Since $p$ avoids one of the boundary segments, we have $k+l \le 6$.

If $p$ starts and ends inside the grid and intersects $k$ big squares before it reaches the boundary of the grid for the first time, $l$ during the following traversal, and $m$ after it enters the grid for the second time, it has to cross at least $k+l+1+m$ coarse grid segments. That gives us $k+l+m \le 7$.

It follows that any pseudosegment can intersect at most $7$ big squares, thus at most $7$ short pseudosegments.
\end{proof}

\subsection{Proof of Theorem~\ref{veta_crossable_nonflat}}

The core of the construction is the arrangement of five segments in Figure~\ref{obr_4_nonflat_core}.

\begin{lemma}\label{lemma_core} 
The arrangement $M_0$ of segments $p_1, \dots, p_5$ crossing a common vertical line $q$ in Figure~\ref{obr_4_nonflat_core}, left, cannot be homeomorphically flattened. More precisely, for a sufficiently small $\varepsilon$, there is no homeomorphism of the plane mapping each segment $p_i$ onto a segment, the line $q$ onto a line, the left endpoint of each segment to an $\varepsilon$-neighborhood of the point $(0,0)$, and the right endpoint of each segment to an $\varepsilon$-neighborhood of the point $(1,0)$.
\end{lemma}

\begin{figure}
\begin{center}
\epsfbox{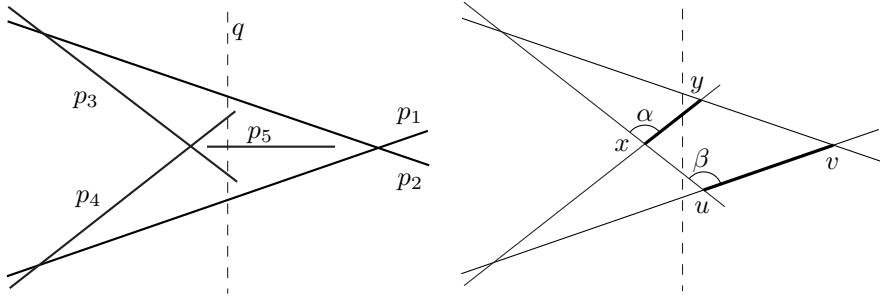}
\end{center}
\caption{Arrangement $M_0$, a core of the construction for Theorem~\ref{veta_crossable_nonflat}.}
\label{obr_4_nonflat_core}
\end{figure}

\begin{proof}
Suppose for contradiction that $M_0$ is already flattened by such a homeomorphism (for sufficiently small $\varepsilon$). Let $x\in p_3\cap p_4$ and $v\in p_1\cap p_2$. Let $y$ be an intersection of $p_2$ with the line extending the segment $p_4$. Similarly, let $u$ be an intersection of $p_1$ with the line extending the segment $p_3$. See Figure~\ref{obr_4_nonflat_core}, right. As all the right endpoints are close to $(1,0)$, the points $y$, $u$ and $v$ are also close to $(1,0)$ since they are to the right from the right endpoint of $p_3$ or $p_4$, and to the left from the right endpoint of $p_1$ or $p_2$. The slopes of all the segments are close to $0$, thus $\beta > \alpha > \pi/2$. It follows that $\| x-y \| < \| u-v \|$, hence $x$ is close to $(1,0)$ as well. 

The segments $p_3$ and $p_4$ and the line $q$ form a triangle $T$, which contains the left endpoint of $p_5$. Since all the vertices of $T$ are close to $(1,0)$, the left endpoint of $p_5$ is close to $(1,0)$ as well, a contradiction.
\end{proof}

By Lemma~\ref{lemma_core}, we only have to add some other segments to $M_0$ so that in any realization of the resulting arrangement $M$ in the plane such that all segments cross a common line $q$, the subarrangement $M_0$ (together with the line $q$) is homeomorphic to the arrangement in Figure~\ref{obr_4_nonflat_core}, left.

We add $18$ segments parallel to $p_2$ and $18$ segments parallel to $p_1$, so that they form an $18 \times 18$ grid as in Figure~\ref{obr_4_nonflat_grid}. All these $38$ segments are called {\em grid\/} segments. As in the construction in the previous section, by taking every odd grid segment we get a coarse $9 \times 9$ grid. These segments are denoted by $g_1, \dots, g_{10}$ and $h_1, \dots, h_{10}$ and drawn by full lines in Figure~\ref{obr_4_nonflat_grid}. We add $17$ {\em short\/} segments to $17$ cells of the coarse grid along the diagonal, each short segment crossing two (even) grid segments. We obtain an arrangement $M_1$ where the intersections between segments are defined by the drawing in Figure~\ref{obr_4_nonflat_grid}.

\begin{figure}
\begin{center}
\epsfbox{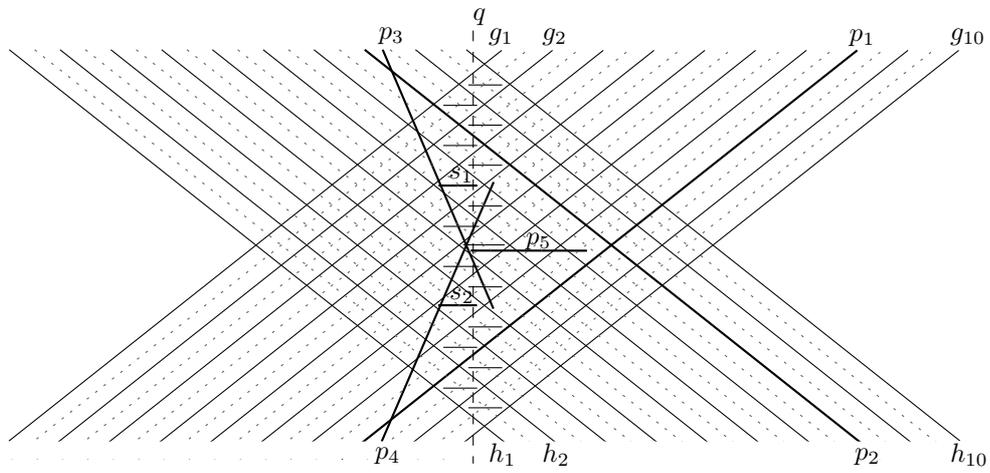}
\end{center}
\caption{Arrangement $M_1$ consisting of core, grid and auxiliary segments.}
\label{obr_4_nonflat_grid}
\end{figure}

To get the final arrangement $M$, we add a frame and some joining segments, as in the construction in the previous section. We add one joining segment for each $p_3$ and $p_4$, and two joining segments for each grid segment. In total, we add $78$ joining segments connected to every other segment of a cycle of length $156$. It is easy to ensure that all the added segments still cross the line $q$; see Figure~\ref{obr_4_nonflat_frame} for an example with smaller grid.

\begin{figure}
\begin{center}
\epsfbox{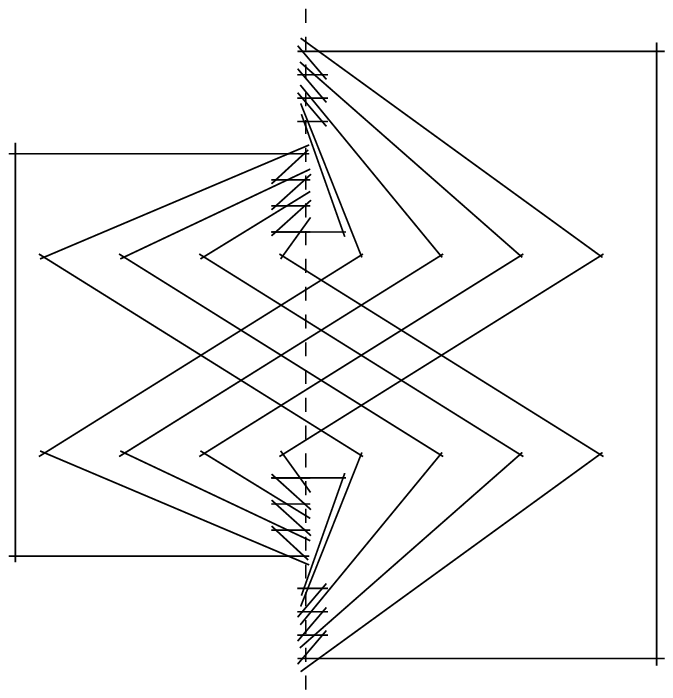}
\end{center}
\caption{An example of the frame and the joining segments added to a small grid arrangement in such a way that all segments cross a common line.}
\label{obr_4_nonflat_frame}
\end{figure}

Now we fix an arbitrary (sufficiently flattened) realization $M'$ of $M$ such that there is a line $q$ crossing all segments from $M'$.

By the same argument as in the previous section, the grid segments form a grid homeomorphic to the grid in Figure~\ref{obr_4_nonflat_grid}. The line $q$ can pass through at most $17$ cells of the coarse grid, since it crosses two of the segments $g_1, g_{10}, h_1, h_{10}$ when entering and leaving the grid, and one other segment $g_i$ or $h_i$ between every two cells in the coarse grid. Each of the $17$ short segments has to lie in the same cell as in Figure~\ref{obr_4_nonflat_grid}. It follows that $q$ passes exactly through these $17$ cells and also in the same order as in Figure~\ref{obr_4_nonflat_grid}. As a consequence we get that the orientation of the segments $g_1, \dots, g_{10}$ and $h_1, \dots, h_{10}$ induced by the grid is consistent with the left-right orientation induced by the line $q$, as in Figure~\ref{obr_4_nonflat_grid}. 

The segment $p_5$ has to lie inside the same $3\times 3$ subgrid of the coarse grid as in Figure~\ref{obr_4_nonflat_grid}. Moreover, since it crosses $q$, it also has to start and end in the same two cells (but the cells it passes through are not uniquely determined). Since both $p_3$ and $p_4$ are connected to the frame between specific pairs of grid segments, one of their endpoints lies outside the grid and the other endpoint lies in the same cell as in Figure~\ref{obr_4_nonflat_grid}. 

We can restrict the position of $p_3$ and $p_4$ even further. Since $p_3$ crosses the short segment $s_1$, it has to pass through the corresponding cell. As a consequence we get that the intersection of $p_3$ with $q$ lies ``below'' the intersection of $p_5$ with $q$, otherwise $p_3$ would cross $p_5$ or cross $h_7$ twice. Similarly, as $p_4$ crosses $s_2$, it has to cross $q$ ``above'' the intersection of $p_5$ with $q$. Also, by the same reason, starting from the endpoint inside of the grid, both $p_3$ and $p_4$ cross $q$ before they cross $p_1$ or $p_2$. Therefore, the sub-arrangement of $p_1, \dots, p_5$ and $q$ in $M'$ is homeomorphic to the arrangement in Figure~\ref{obr_4_nonflat_core} and the proof of Theorem~\ref{veta_crossable_nonflat} is finished.

\section*{Acknowledgments}
The author is grateful to G\'eza T\'oth and Jakub \v{C}ern\'y for helpful and inspiring discussions on this problem.

\end{document}